\documentclass{llncs}
\usepackage{makeidx}  
\usepackage{amsmath}
\usepackage{amsfonts,amssymb}
\usepackage[thinlines,thiklines]{easybmat}
\usepackage{multirow}
\usepackage{algorithmic}
\usepackage{algorithm}
\usepackage{url}

\newcommand{\Z}{\mathbb{Z}} 
\newcommand{\Q}{\mathbb{Q}} 
\newcommand{\R}{\mathbb{R}} 
\newcommand{\K}{\mathcal{K}} 
\newcommand{\p}{\mathfrak{p}} 
\newcommand{\M}{\mathcal{M}} 
\newcommand{\OO}{\mathcal{O}_{\mathbb{K}}}
\newcommand{\ag}{\mathfrak{a}} 
\newcommand{\bg}{\mathfrak{b}} 
\renewcommand{\algorithmicrequire}{\textbf{Input:}}
\renewcommand{\algorithmicensure}{\textbf{Output:}}

\newtheorem{heuristic}{Heuristic}
\begin{document}

\mainmatter              
\title{An $L(1/3)$ algorithm for discrete logarithm computation and principality testing in certain number fields}
\titlerunning{An $L(1/3)$ algorithm for certain number fields}  
%
\author{Jean-Fran\c{c}ois Biasse\inst{1}}
\authorrunning{J-F Biasse}   
%
\tocauthor{Jean-Fran\c{c}ois Biasse}
\institute{ \'{E}cole Polytechnique , 91128 Palaiseau , France\\
\email{biasse@lix.polytechnique.fr}
}

\maketitle              

\begin{abstract}
We analyse the complexity of solving the discrete logarithm problem and of testing the principality of ideals in a certain class of number fields. We achieve the subexponential complexity in $O(L(1/3,O(1)))$ when both the discriminant and the degree of the extension tend to infinity by using techniques due to Enge, Gaudry and Thom\'{e} in the context of algebraic curves over finite fields. 
\end{abstract}

\section{Introduction}

Quadratic number fields were proposed as a setting for public-key
cryptosystems in the late 1980s by Buchmann and Williams
\cite{BWKeyEx,BWKeyExReal}. Those cryptosystems were generalized to number fields of arbitrary dimension about a decade later \cite{arb_dim_1,arb_dim_2,arb_dim_3}. Their security relies on the hardness of the discrete logarithm problem and the principality testing problem. The complexity of the algorithms for solving these problem on a number field $\K$ of discriminant $\Delta$ is bounded by $L(1/2,O(1))$, where the subexponential function is defined as
$$L(\alpha,\beta ) = e^{\beta\log_2|\Delta|^{\alpha}\log_2\log_2|\Delta|^{1-\alpha}}.$$
This complexity is asymptotically slower than the one for factoring which reduces to the problem of computing the class number, and although the discrete logarithm problem in the Jacobian of elliptic curves remains exponential, there is no known reduction between this problem and the discrete logarithm problems in number fields either. Therefore, studying the hardness of the discrete logarithm problem and of the principality testing problem on number fields is of cryptographic interrest since they provide alternative cryptosystems whose security is unrelated to those currently being used.\\
In this paper, we exhibit the first infinite class of number fields for which these problems can be solved in expected time bounded by $L(1/3 , O(1) ) $. We follow the approach of Biasse \cite{biasseL13} who described a class of number fields on which class group and regulator computation can be done in expected time $L(1/3,O(1))$, and the one of Enge, Gaudry and Thom\'{e} \cite{Enge,Enge2} who described and algorithm for solving the discrete logarithm problem in complexity $L(1/3 , O(1))$ in certain algebraic curves.\\

\section{Number fields}

Let $\K$ be a number field of degree $n$, $\theta\in\K$, and $T[X] = \sum_{i\leq n} t_iX^i\in\Z[X]$ such that 
$$\K = \Q[X]/T(X) = \Q(\theta).$$
We denote by $\OO$ its maximal order and by $Cl(\OO)$ the ideal class group of its maximal order. The ideal class group of an order is a finite group of cardinality denoted by $h(\OO)$ which is unknown to both parties in number field cryptosystems. Solving the discrete logarithm problem with respect to $\mathfrak{a}$ and $\mathfrak{b}\in Cl(\OO)$ consists of finding $x\in\Z$ such that
$$\mathfrak{b} = \mathfrak{a}^x.$$
The principality testing problem with respect to an ideal $I$ of $\OO$ consists of deciding if there exists $\alpha\in\OO$ such that
$$I = (\alpha),$$
and if so, computing $\alpha$. Direct computation of $\alpha$ in subexponential time is impossible because of the size of its coefficients, thus obliging us to give a compact representation of this value, that is to say a vector $\overrightarrow{v}=(v_1,\hdots,v_k)$ and $\gamma_1,\hdots,\gamma_k\in\K$ satisfying 
$$\alpha = \gamma_1^{v_1}\hdots\gamma_k^{v_k}.$$
In number fields of fixed degree (typically when the dimension is 2), these problems can be solved in subexponential time. The strategy described in \cite{Bsub} consists of defining a factor base $\mathcal{B}$ containing the primes of norm bounded by an integer $B$ and reduce random power-products $\p_1^{e_1}\hdots\p_g^{e_g}$ of elements $\p_i\in\mathcal{B}$ untill an equivalent $\mathcal{B}$-smooth ideal is found. Whenever this occurs, we can derive a row of the so-called relation matrix which after a suitable linear transformation yields the structure of $Cl(\OO)$, and enables us to solve instances of the discrete logarithm problem and principal ideal problem. If the degree is no longer assumed to be fixed, then every reduction step is exponential in the degree of $\K$ since it uses the LLL algorithm \cite{LLL}.

\section{Main idea}

Let $d := \max_i \left\lbrace \log_2(t_i)\right\rbrace$, we require that
\begin{align}\label{cond_n}
 &n= n_0\log_2\left( |\Delta|\right)^{\alpha}(1+o(1))\\
& d= d_0\log_2\left( |\Delta|\right)^{1-\alpha}(1+o(1)),\label{cond_d}
\end{align}
for some $\alpha\in\left[ \frac{1}{3},\frac{2}{3}\right[$, and  some constants $n_0$ and $d_0$. We define  $\kappa:=n_0d_0$. We also denote by $s$ the number of real places, by $t$ the number of complex places and we define $r:=t+s-1$. We also require that  $\Z[\theta] = \OO$.

\subsubsection*{Example}

Let $\Delta\in\Z$, and $\K_{n,K}$ be an extension of $\Q$ defined by a polynomial of the form: 
$$T(X) = X^n - K,$$
with 
\begin{align*}
&\log K = \left\lfloor \log_2\left( |\Delta|\right)^{1-\alpha}\right\rfloor\\
& n= \left\lfloor\log_2\left( |\Delta|\right)^{\alpha}\right\rfloor,
\end{align*}
for some $\alpha\in\left[ \frac{1}{3},\frac{2}{3}\right[$. Then, $\mathcal{O}_{\K_{n,K}}$ has discriminant satisfying: 
$$\log_2(\text{Disc}(\mathcal{O}_{\K_{n,K}}))=\log( n^{n}K^{n-1}) = \log_2(|\Delta|) (1+o(1)).$$
If in addition we require that $n$ and $K$ be the largest prime numbers below their respective bounds such that: 
$$n^2 \nmid K^{n-1}-1,$$
then we meet the last restriction $\Z[\theta] = \mathcal{O}_{\K_{n,K}}$ (for a proof, see \cite{cohen}, Chapter 6 \textsection 1).\\

In \cite{biasseL13}, it is shown that the computation of the group structure of $Cl(\OO)$ and of the regulator of $\OO$ with a number of bits of precision in $L(1/3,O(1))$ could be achieved in expected time $L(1/3,O(1))$ under some assumption that we will specify in the following. The main idea is to use sieving based technique to create relations of the form
$$(\phi) = \p_1^{e_1}\hdots\p_n^{e_n},$$
where $\phi\in\OO$ and the $\p_i$ are non inert prime ideals of norm bounded by a certain integer $B$; we denote this set by $\mathcal{B}$. Every time such a relation is found, the vector
$$(e_1,\hdots,e_n,\log|\phi|_1,\hdots,\log|\phi|_r)$$
is added as a row of the relation matrix $M$, which has the following shape
\[M=
\left( 
\begin{BMAT}[2pt,3cm,1cm]{c.c}{c} 
M_{\Z} & M_{\R}
\end{BMAT}
\right).  \]
Then, provided the rows of $M$ generate the whole lattice of relations, the Smith normal form of $M_{\Z}$ yields the group structure of $Cl(\OO)$ whereas its kernel yields $R$.\\
Now, given two ideals $\mathfrak{a}$ and $\mathfrak{b}$ such that $\exists x\in\Z\ \bg = \ag^x$, computing their discrete logarithm can be done by decomposing them over $\mathcal{B}$,
\begin{align*}
 \mathfrak{a} &= \p_1^{e_1}\hdots \p_n^{e_n} \\
\mathfrak{b} &= \p_1^{f_1}\hdots \p_n^{f_n},
\end{align*}
and performing a linear algebra phase consisting of solving one linear system. Likewise, if we need to test the principality of an ideal $I$ and compute $\alpha$ such that $I = (\alpha)$, then it sufficies to find $b:=[e_1,\hdots,e_n]$ such that
$$I = \p_1^{e_1}\hdots\p_n^{e_n}.$$
$I$ is principal if and only if $b$ belongs to the lattice of relations. We thus solve $XM_{\Z} = b$ and derive $\alpha$ from the coefficients of $X$ and the generators $\phi_i$ of the relations of $M$. We thus see here that solving the discrete logarithm problem and testing the principality rely on our ability to decompose an arbitrary ideal into a power product of elements of $\mathcal{B}$.\\
To do this, we follow the approach of Enge, Gaudry and Thom\'{e} for algebraic curves \cite{Enge,Enge2} involving a $Q$-descent strategy. Given an ideal $I$, it consists of decomposing it as a power product of prime ideals (not necessarily in $\mathcal{B}$), and then decomposing those primes as power products of primes of a lower norm untill we only have prime ideals of norm bouned by $B$.

\section{Relation matrix}

Let $\rho$ be a constant to be determined later, and $B$ a smoothness bound satisfying: 
$$B = \lceil L(1/3,\rho)\rceil.$$ 
We define the factor base $\mathcal{B}$ as the set of all non inert prime ideals of norm bounded by $B$. This factor base has cardinality: 
$$N:=|\mathcal{B}| = L(1/3,\rho+o(1)).$$
The sieving phase consists of enumerating $\phi\in\OO$ of the form
$$\phi = A(\theta),$$
with $A[X]\in\Z[X]$ of degree $k$ whose coefficients $a_i$ have their logarithm bounded by an integer $a$ such that there exist two constants $\delta$ and $\nu$ to be determined later satisfying: 
\begin{align}
 a &\leq \left\lceil \delta \frac{\kappa\log_2|\Delta|/n}{(\log_2|\Delta|/\mathcal{M})^{1/3}} \right\rceil \label{bound_a}\\
 k &\leq \left\lceil \nu \frac{n}{(\log_2|\Delta|/\mathcal{M})^{1/3}} \right\rceil, \label{bound_k}
\end{align}
with $\mathcal{M}:=\log_2\log_2|\Delta|$. Landau-Mignotte's theorem \cite{mignotte} states that if $D\mid T$ with $\deg D=m$, then the coefficients $d_j$ of $D$ satisfy $|d_j|\leq 2^{m-1}(|T| + t_n)$, where $|T|$ is the euclidian norm of the vector of the coefficients of $T$. Applying this to $D = X-\sigma_i(\theta)$ and $m=1$ allows us to obtain:
$$\log(|\theta|_i)\leq \log(|T|+t_n) \in O(\log\left( |\Delta|\right) ^{1-\alpha}),$$
for $i\leq r$. From $\phi = A(\theta)$, and $a$ and $k$ respectively bounded by (\ref{bound_a}) and (\ref{bound_k}), we have
$$\log|\phi|_i \leq O(\log\left( |\Delta|\right) ^{2/3}\M^{1/3}).$$


We can thus derive a bound on the maximum value $|M_\Z|$ of the norm a coefficient of $M_\Z$.
\begin{proposition}\label{bound_B}
$ |M_{\Z}|$ satisfies:
$$ |M_{\Z}| = O(\left( \log_2|\Delta|\right) ^{2/3}\left( \log_2\log_2|\Delta|\right) ^{1/3}).$$
\end{proposition}

During the relation collection phase, we collect $N+Kr$ relations, where $K$ is a constant. We rely on the following heuristic to make sure that we generate the full lattice of relations.

\begin{heuristic}\label{heuristic:lattice}
The  $N+Kr$ relations collected this way generate the full lattice of relations.
\end{heuristic}

\section{Smoothness}\label{sec:smooth}

We need to evaluate the smoothness of ideals with respect to $\mathcal{B}$. Let $\psi_{\mathcal{I}}(\iota,\mu)$ be the set of ideals $I$ such that $\mathcal{N}(I)\leq \iota$ which are smooth with respect to the set of primes $\p$ satisfying $\mathcal{N}(\p)\leq \mu$ and $\psi(x,y)$ be the set of integers of logarithm bounded by $x$ smooth with respect to primes of logarithm bouded by $y$. $\psi$ was first described in \cite{Canfield} by Canfield, Erd\"{o}s Pomerance. We need to make the following assumption on the smoothness of ideals.
\begin{heuristic}\label{heuristic:norm}
We assume that
\begin{equation}\label{eq:smooth_ideal}
\frac{\psi(\iota,\mu)_{\mathcal{I}}}{e^\iota}\geq\exp\left( -u \left( \log_2 u + \log_2\log_2 u -1 + O\left( \frac{\log_2\log_2 u}{\log_2 u}\right) \right) \right), 
\end{equation}
for $u = \iota / \mu$. In addition, assume that $\mathcal{N}(\phi)$ behaves like a random number whose logarithm satisfies
$$\log_2(\mathcal{N}(\phi))\leq \iota:=\kappa\log_2\left( |\Delta|\right) ^{2/3}\mathcal{M}^{1/3}(\delta + \nu + o(1) ),$$
and whose distribution is given by 
\begin{equation}\label{eq:smooth_principal}
\frac{\psi(\iota,\mu)}{e^\iota}\geq\exp\left( -u \left( \log_2 u + \log_2\log_2 u -1 + O\left( \frac{\log_2\log_2 u}{\log_2 u}\right) \right) \right).
\end{equation}
\end{heuristic}



The assertion concerning $\psi_{\mathcal{I}}$ can be proved in the quadratic case\cite{seysen} but remain conjectural for arbitrary $n$ \cite{Bsub}. In the context of curves, Enge, Gaudry and Thom\'{e} used a theorem due to Hess to derive the equivalent of \eqref{eq:smooth_ideal} for divisors in the jacobian of a curve, but had to use a similar heuristic for \eqref{eq:smooth_principal}.
 Using \cite{Canfield}, and carrying out the same computation as in the proof of theorem 1 of \cite{Enge,Enge2}, one readily shows the following result on the probability of finding a relation:
\begin{proposition}\label{smoothness}
Let :
\begin{align*}
\iota&= \lfloor\log_2 L(\phi , c)\rfloor = \lfloor c\log_2\left( |\Delta|\right) ^{\phi}\M^{1-\phi}\rfloor \\
\mu&=  \lceil\log_2 L(\beta,d)\rceil= \lceil d\log_2\left( |\Delta|\right) ^{\beta}\M^{1-\beta}\rceil,
\end{align*}
then we have 
\begin{align*}
\frac{\psi(\iota,\mu)_\mathcal{I}}{e^{\nu}}& \geq L\left( \phi-\beta,\frac{-c}{d}(\phi-\beta)+o(1)\right)\\
\frac{\psi(\iota,\mu)}{e^{\nu}} & \geq L\left( \phi-\beta,\frac{-c}{d}(\phi-\beta)+o(1)\right)
\end{align*}
\end{proposition}

Proposition \ref{smoothness} allows us to bound the expected time for finding a $\mathcal{B}$-smooth ideal. In \textsection \ref{sec:decomp}, we show how to decompose prime ideals of the form $p\OO + (\theta - v_p)\OO$ over a set of prime ideals of the same form with a smaller norm. In the general case, prime ideals can have a ramification index $f\geq 2$ and thus be of the form $p\OO + T_p(\theta)\OO$ where $\deg(T_p) = f$. However, it can be shown that the ramified primes have Dirichlet density 0, allowing us to consider that $\mathcal{B}$-smooth decomposition with unramified primes occur with the same probability as in Proposition \ref{smoothness}. A proof of this result can be found in Chapter IV, Proposition 4.5 of \cite{janusz} for example.

Proposition \ref{smoothness} with parameters $\beta = \frac{1}{3}$, $ d = \rho$, $\phi = \frac{2}{3}$ and $c = \kappa(\delta+\nu+o(1))$
shows that the expected number of trials to obtain a relation is at most $L\left( 1/3,\frac{\kappa(\nu+\delta)}{3\rho}+o(1)\right)$. Since the factor base has size $N\in O\left( L(1/3,\rho)\right) $, the complexity of the relation collection phase with respect to the parameters $\rho,\nu,\kappa,\delta$ is in
$$L\left( 1/3,\frac{\kappa(\nu+\delta)}{3\rho}+\rho+o(1)\right).$$
 These parameters are chosen to ensure that the overall time be optimal. The linear algebra phase is polynomial in the dimension of $M$ wich is given by $L(1/3,\rho + o(1))$. We need to compute the regulator, which can be done in expected time $L(1/3,3\rho + o(1))$ provided the bit precision is also bounded by $L(1/3,3\rho + o(1))$ (see \cite{biasseL13}). It is shown in \cite{JaSto} that linear systems of the form $X M_\Z$ can be solved in time 
$$O\left( N^3(\log_2 n + \log_2|M_\Z|)^2\right),$$
where $|M_\Z|$ is the largest absolute value of a coefficient of $M_\Z$. The computation of a discrete logarithm in $Cl(\OO)$ with Vollmer's method \cite{Vdl} is done by solving a system of the form $XM_\Z'$ where $M_\Z'$ is $M_\Z$ augmented with two extra rows whose coefficients are proved to be bounded by $e^{o(\log_2|\Delta|^{1/3}\mathcal{M}^{2/3})}$ in \textsection \ref{sec:DLP}. The linear algebra phase thus has a complexity bounded by $L(1/3,3\rho + o(1))$. We emphasize here that we do not need to compute the group structure of $Cl(\OO)$, thus avoiding the computation of the Hermite Normale Form of $M_\Z$. We can prove that the optimal strategy is to spend the same amount of time for the relation collection and for the linear algebra. Therefore, the parameters must satisfy
\begin{equation}\label{constr1}
\kappa\nu\delta = 3\rho. 
\end{equation}
In addition, the number of $\phi$ in the search space is in $O\left( L(1/3) , \nu\delta\kappa\right)$. We thus have the additional constraint on the parameters
\begin{equation}\label{constr2}
\nu\delta\kappa = \frac{\kappa(\nu+\delta)}{3\rho}+\rho,	
\end{equation}
ensuring that the search space is large enough to yield the $N+Kr$ relations. From \eqref{constr1} and \eqref{constr2}, we obtain 
\begin{align*}
\nu\delta &= \frac{3\rho}{\kappa} \\
\nu + \delta &= \frac{6\rho^2}{\kappa}.
\end{align*}
Thus, $\delta$ and $\nu$ are roots of the polynomial
$$X^2 - \frac{6\rho^2}{\kappa}X + \frac{3\rho}{\kappa}.$$
These roots exist provided we have
$$\rho\geq \sqrt[3]{\frac{\kappa}{3}}.$$
The optimal choice is to minimize $\rho$, thus fixing the parameters $\delta$ and $\nu$:
$$\delta = \nu = \sqrt{\frac{3\rho}{\kappa}}.$$
The total running time becomes $L(1/3 ,c + o(1))$, with
$$c = 3\rho= \sqrt[3]{9\kappa}.$$

\section{Decomposition over $\mathcal{B}$}\label{sec:decomp}

Assuming Heuristics \ref{heuristic:lattice} and \ref{heuristic:norm}, we can study the complexity of the $Q$-descent. In what follows, we show how to decompose an ideal as a power product of elements of $\mathcal{B}$ starting with a lemma allowing us to find integers $\alpha_1,\hdots,\alpha_{k+1}$ minimizing $\sum_i\alpha_i v_i$ for some $v_i$.

\begin{lemma}\label{lemma:smooth}
Let $v_1,\hdots,v_{k+1}$ be integers satisfying $\log|v_i|\leq D$ for some integers $D$ and $k$ defined by
$$k := \left\lfloor \sigma \frac{n}{(\log_2|\Delta|/\mathcal{M})^{1/3-\tau/2}}\right\rfloor\ \ D := \log_2\left( L(1/3+\tau,c)\right) ,$$
where $\sigma ,\tau, c > 0$. Then for any integer $z$, there exist at least $2^{kz}$ $(k+1)$-tuples $(\alpha_1,\hdots,\alpha_{k+1})$ satisfying
\begin{align*}
\log_2|\alpha_i|&\leq D/k + z \\
\log_2\left|\sum_i \alpha_i v_i \right| &\leq D/k + z.
\end{align*}
\end{lemma}

\begin{proof}
Let us define the $k+1$ dimensional lattice $\Lambda$ generated by the rows of
\[ A:=\left( \begin{array}{ccccc}
1 & 0      & \hdots & 0  & v_1    \\
0 & 1& \ddots &\vdots & \vdots \\
\vdots & \ddots & \ddots & 0 & \vdots     \\
0      &  \hdots     & 0& 1 & v_{k+1}\\
\end{array} \right).\] 
For any element $x\in\Lambda$, there exist $(\alpha_1,\hdots,\alpha_{k+1})\in\Z^{k+1}$ such that
$$x = (\alpha_1,\hdots,\alpha_{k+1},\sum_i\alpha_i v_i).$$
The determinant $d(\Lambda)$ of $\Lambda$ satisfies
$$d(\Lambda)= \sqrt{\det\left( AA^T\right)}= \sqrt{\sum_{i\leq k+1}v_i + \sum_{i\leq k+1} v_iv_{k+1-i}}\leq\left( \sqrt{2k+1}\right) 2^D.$$ 
Let $X\subset\R^{k+2}$ be the symmetric and convex set of points defined by 
$$X = \left\lbrace (x_1,\hdots,x_{k+2})\mid \forall i\ |x_i| \leq D/k + z\right\rbrace .$$
The volume $V(X)$ equals $2^{k+2}e^{(k+2)(D/k+z)}$, and from Theorem II of III.2.2 in \cite{cassel} we know that if
$$V(X) > m 2^{k+2}d(\Lambda),$$
then $X$ intersect $\Lambda$ in at least $m$ pairs of points $\pm x\in\R^{k+2}$. It thus suffices to prove that 
$$2^{kz} < \frac{2^{(k+2)(\frac{D}{k}+z)}}{\sqrt{2k+1}e^D} = 2^{kz}.\frac{2^{2\frac{D}{k} + 2z}}{\sqrt{2k+1}},$$
which is satisfied since 
$$\frac{D}{k} = \frac{c}{\sigma}\log_2|\Delta|^{2/3 - \alpha  +\tau/2}\log_2\log_2|\Delta|^{1/3-\tau/2}\gg \log_2 (2k+1).$$
\end{proof}

Using Lemma \ref{lemma:smooth}, we can state the analogue of Theorem 8 in \cite{Enge2}. Please note here that the proof we give is almost verbatim, the main difference being the use of Lemma \ref{lemma:smooth}.

\begin{theorem}\label{theo:smooth}
Assuming Heuristic \ref{heuristic:norm}, we can decompose any ideal $I$ of $\OO$ into a power product of elements of $\mathcal{B}$ in time
$$L(1/3 , b + \varepsilon),$$
with $b = \sqrt[3]{\frac{24\kappa}{9}}$ and any $\varepsilon > 0$.
\end{theorem}

\begin{proof}
Let $I$ be an ideal of norm bounded by $|\Delta|$. We can assume this without loss of generality since any class of $Cl(\OO)$ contains an ideal of norm bounded by $(2/\pi)^s\sqrt{\Delta}$Let $I = u\OO + (\theta - v)\OO$ be an ideal of norm bounded by $ L(1/3+\tau,c)$ for some $c>0$ and $0\leq \tau \leq 2/3$. The ideal we start has $\tau = 2/3$ and $c=1$. Indeed, it can be proved that any class of $Cl(\OO)$ contains an ideal of norm bounded by $|\Delta|$. We search a $L(1/3 + \tau/2,c')$-smooth $\phi\in I$ for a $c'$ depending on $c$. Such a $\phi$ satisfies $I\mid (\phi)$ and thus $I$ can be decomposed as a power product of the prime ideals involved in the decomposition of $(\phi)$. We repeat this process untill we obtain a decomposition only involving elements of $\mathcal{B}$. At each stage, we consider $\phi$ belonging to the lattice of polynomials of degree bounded by 
$$k := \left\lfloor \sigma \frac{n}{(\log_2|\Delta|/\mathcal{M})^{1/3-\tau/2}}\right\rfloor,$$
where $\sigma > 0$ is a constant to be determined later. These $\phi$ form a $\Z$-lattice generated by
$$(v_0,\theta - v_1 , \hdots , \theta^k - v_k),$$
with $v_0 = u$ and $v_i = v^i\mod u$ for $i\geq 1$. We want to spend the same time $L(1/3,e+o(1))$ at each smoothing step for $e>0$ to be optimised later. The sieving space has to be of the same size. We thus look for $L(1/3,e+o(1))$ distinct $(k+1)$-tuples $(\alpha_1,\hdots,\alpha_{k+1})\in\Z^{k+1}$. Using Lemma \ref{lemma:smooth}, we prove that for every integer $z$, we can find $2^{kz}$ such tuples satisfying $\log_2|\alpha_i|\leq D/k + z$ for $i\leq k+1$ and $\log_2\left|\sum_i \alpha_i v_i \right| \leq D/k + z$. We ajust the value of $z$ to make sure that all the $L(1/3,e+o(1))$ obtained during the sieving phase satisfy this property by solving $2^{kz} = L(1/3, e + o(1))$. This yields
$$z = \frac{1}{n}\log_2 L(2/3 - \tau/2,e/\sigma + o(1)).$$
Carrying on the same computation as in \cite{Enge,Enge2}, we can prove that the norm of the $\phi$ we create during the sieving phase satisfies
$$\mathcal{N}(\phi)\leq  L(2/3 + \tau/2 , (c+e)/\varphi + o(1)).$$
From Heuristic \ref{heuristic:norm} and Proposition \ref{smoothness} we expect to find at least one $ L(1/3+\tau/2 ,c')$-smooth $\phi$ for 
$$c' = \frac{1}{3e}((c+e)/\sigma + \sigma\kappa).$$
This quantity is minimised by $\sigma = \sqrt{(c+e)/\kappa}$ which yields
$$c' = \frac{2\sqrt{\kappa}}{3e}\sqrt{c+e}.$$
Starting with $\tau_0 = 2/3$ and $c_0 = 1$, we obtain a power-product of places of norm bounded by $L(1/3+\tau_1 , c_1)$ with $\tau_1 = 1/3$ and $c_1 = 2\sqrt{\kappa(c_0+e)}/3e$. After $i$ step we get an ideal $L(1/3 + 1/(3.2^{i-1}),c_i)=L(1/3,c_i\mathcal{M}^{\frac{1}{3.2^{i-1}}})$-smooth where
$$\tau_i = \frac{1}{3.2^{i-1}},\ \ c_i = \frac{2\sqrt{\kappa}}{3e}\sqrt{c_{i-1}+e}.$$
The sequence $c_i$ converves to a finite limit $c_\infty$ given by
$$c_\infty = \chi/2\left( \chi + \sqrt{\chi^2 + 4e}\right) ,$$
where $\chi = 2\sqrt{\kappa}/3e$. Let $\xi>0$ be an arbitrary constant. Afer a number of steps only depending on $e$, $\kappa$ and $\xi$, we have $c_i < c_\infty(1+\xi)$, and after $O(\log_2\log_2|\Delta|)$ steps $ \mathcal{M}^{\frac{1}{3.2^{i-1}}} < (1+\xi)$. We can thus decompose $I$ as a power-product of prime ideals of norm bounded by
$$L(1/3 , c_\infty (1+\xi)).$$
As each node of the tree has arity $\log_2|\Delta|$, the number of nodes in the tree is in $L(1/3,o(1))$ and the complexity of the algorithm is in $L(1/3,e+o(1))$. As we want to decompose $I$ as a power product of primes of norm bounded by $L(1/3,\rho)$, we compute the effort to reach $c_\infty = \rho$. As in \cite{Enge,Enge2}, we write $9e^{1/3} = E\kappa$ for $E$ with $E$ to be determined later. The equation $\rho = c_\infty$ simplifies as
$$\left( \frac{3}{E}\right)^{1/3} = \frac{2}{E}(1+ \sqrt{1+E}).$$
The least non negative solution $E_0$ satifies $E_0\approx 24$, which yields
$$e = \sqrt[3]{\frac{24\kappa}{9}}=:b$$
\end{proof}

The time taken to decompose an ideal over $\mathcal{B}$ is subexponential with a constant $b + \varepsilon$ stricly lower than the one minimizing the time taken by the relation collection and the linear algebra (see \textsection \ref{sec:smooth}). Therefore, there is no need for a more elaborated optimization of the parameters encapsulating the time for decomposing an ideal over $\mathcal{B}$.

\section{Discrete Logarithm algorithm and principality testing}\label{sec:DLP}

We follow the approach of Vollmer in quadratic fields\cite{Vdl} to compute discrete logarithms without computing the group structure of $Cl(\OO)$. Given two ideals $\ag$ and $\bg$ such that there exists an integer $x$ satisfying $\bg = \ag^x$, we wish to compute $x$. We enlarge the factor base with $\ag$ and $\bg$ and let $\mathcal{B}' = \mathcal{B}\cup\left\lbrace \ag,\bg\right\rbrace $. Then we use the methods of \textsection \ref{sec:decomp} to decompose $\ag$ and $\bg$ over $\mathcal{B}$, thus creating two extra relations over $\mathcal{B}'$
\begin{equation}\label{eq:decomp_ab}
\p_1^{e_1}\hdots\p_N^{e_N}\ag=\alpha_\ag,\ \ \  \p_1^{f_1}\hdots\p_N^{f_N}\bg=\alpha_\bg.
\end{equation}
Then we construct the extended relation matrix 
\[ M_\Z' := \left( 
   \begin{BMAT}(@)[2pt,1.5cm,1.5cm]{c.c}{c.c}
   \begin{BMAT}(2pt,1cm,1cm){c}{c}
M_\Z
  \end{BMAT} &
\begin{BMAT}(2pt,0.5cm,1cm){c}{c}
(0)
  \end{BMAT} \\
\begin{BMAT}(2pt,1cm,0.5cm){c}{cc} 
	\overrightarrow{v_{\mathfrak{b}}} \\
	\overrightarrow{v_{\mathfrak{a}}}
\end{BMAT} &
\begin{BMAT}(2pt,0.5cm,0.5cm){cc}{cc} 
	1 & 0 \\
	0 & 1
\end{BMAT} 
\end{BMAT}
   \right),
\]
where $\overrightarrow{v_{\mathfrak{a}}} = (e_1,\hdots,e_N)$ and $\overrightarrow{v_{\mathfrak{b}}} = (f_1,\hdots,f_N)$. The relation $\bg = \ag^x$ corresponds to the row vector $\overrightarrow{v_{x}}:=(0,\hdots,1,-x)$ which is a combination of the rows of $M_\Z'$ under Heuristic \ref{heuristic:lattice}. Therefore, there exists $X=(x_1,\hdots,x_{N + rK})$ such that $XM_\Z' = \overrightarrow{v_{x}}$. In particular $x_{N+Kr} = -x$. We can thus obtain $x$ by solving $XA = \overrightarrow{v}$ where

\[ A := \left( 
   \begin{BMAT}(@)[1pt,1cm,1cm]{c.c}{c.c}
   \begin{BMAT}(2pt,1cm,1cm){c}{c}
M_\Z
  \end{BMAT} &
\begin{BMAT}(1pt,0.25cm,0.75cm){c}{c}
(0)
  \end{BMAT} \\
\begin{BMAT}(1pt,0.75cm,0.25cm){c}{cc} 
	\overrightarrow{v_{\mathfrak{b}}} \\
	\overrightarrow{v_{\mathfrak{a}}}
\end{BMAT} &
\begin{BMAT}(1pt,0.75cm,0.25cm){c}{cc} 
	1  \\
	0 
\end{BMAT} 
\end{BMAT}
   \right)\ \ \text{and}\ \  \overrightarrow{v} = (0,\hdots,0,1).
\]

It is shown in \cite{JaSto} that the complexity of this step in in 
$$O\left( N^3(\log_2 n + \log_2|A|)^2\right),$$
where $|A| = \max|a_{ij}|$. We already know a bound on the norm of the coefficients of $M_\Z$, but we still have to bound those of $\overrightarrow{v_{\mathfrak{a}}}$ and $\overrightarrow{v_{\mathfrak{b}}}$.

\begin{lemma}\label{lemma:size_coeff}
The size of the coefficients of $\ag$ and $\bg$ is bounded by $O\left(\log_2|\Delta|^{\log_2\log_2|\Delta|}\right)$.
\end{lemma}

\begin{proof}
At each smoothing step, an ideal $I$ of norm satisfying $\mathcal{N}(I)\leq O(\log_2|\Delta|)$ is smoothed. The largest possible exponent $e$ of this decomposition occurs if $I = \p_1^e$ thus yielding 
$$e\leq \frac{\mathcal{N}(I)}{\mathcal{N}(\p_1)}\leq \frac{\mathcal{N}(I)}{2}\in O(\log_2|\Delta|).$$
The depth of the tree is bounded by $O\left( \log_2\log_2|\Delta||\right)$, so the size of the maximal coefficient occuring in the decomposition of $\ag$ and $\bg$ is bouned by $O\left(\log_2|\Delta|^{\log_2\log_2|\Delta|}\right)$.
\end{proof}

We know from Proposition \ref{bound_B} that $ |M_{\Z}| = O(\left( \log_2|\Delta|\right) ^{2/3}\left( \log_2\log_2|\Delta|\right) ^{1/3})$, allowing us to conclude that the overall expected time of the discrete logarithm algorithm is bounded by $L(1/3,3\rho + o(1))$ where $\rho =  \sqrt[3]{\frac{\kappa}{3}}.$

\begin{proposition}
Let $\ag$ and $\bg$ be ideals such that there exists $x\in\Z$ satisfying $\bg = \ag^x$. Under Heuristics \ref{heuristic:lattice} and \ref{heuristic:norm}, the expected time to compute $x$ is in
$$L\left( 1/3 , 3\rho +o(1)\right),$$
where  $\rho =  \sqrt[3]{\frac{\kappa}{3}}$.
\end{proposition}
	

Now let us study how we can decide whether a given arbitrary ideal $I$ is principal, and if so compute  
$\alpha$ such that  $I = (\alpha)$. To this end, we first decompose $I$ over $\mathcal{B}$ using Theorem \ref{theo:smooth}. We thus obtain a vector $b\in\Z^N$ representing the decomposition of $I$ over $\mathcal{B}$. As we assume Heuristic \ref{heuristic:lattice}, $b$ belongs to the lattice spanned by the rows of $M_Z$ if and only if $I$ is principal. Therefore, solving $XM_\Z = b$ allows us to decide whether $I$ is principal. 
Using the same strategy as for the analysis of the discrete logarithm problem algorithm, we can prove that this step has complexity $L(1/3,3\rho + o(1))$.
\begin{proposition}
Under Heuristics \ref{heuristic:lattice} and \ref{heuristic:norm}, the expected time to decide if $I$ is principal and to compute a compact representation if $\alpha$ such that $I = (\alpha)$ is bounded by
$$L\left( 1/3 , 3\rho +o(1)\right),$$
where  $\rho =  \sqrt[3]{\frac{\kappa}{3}}$.
\end{proposition}
\bibliography{qfdlp}

\providecommand{\bysame}{\leavevmode\hbox to3em{\hrulefill}\thinspace}
\providecommand{\MR}{\relax\ifhmode\unskip\space\fi MR }
\providecommand{\MRhref}[2]{%
  \href{http://www.ams.org/mathscinet-getitem?mr=#1}{#2}
}
\providecommand{\href}[2]{#2}
\begin{thebibliography}{10}

\bibitem{biasseL13}
J-F. Biasse, \emph{An $l(1/3)$ algorithm for ideal class group and regulator
  computation in certain number fields}, Submitted to \textit{Mathematics of
  Computation}.

\bibitem{arb_dim_1}
I.~Biehl, J.~Buchmann, S.~Hamdy, and A.~Meyer, \emph{A signature scheme based
  on the intractability of computing roots}, Des. Codes Cryptography
  \textbf{25} (2002), no.~3, 223--236.

\bibitem{Bsub}
J.~Buchmann, \emph{A subexponential algorithm for the determination of class
  groups and regulators of algebraic number fields}, S\'{e}minaire de
  Th\'{e}orie des Nombres (Paris), 1988-89, pp.~27--41.

\bibitem{arb_dim_2}
J.~Buchmann and S.~Paulus, \emph{A one way function based on ideal arithmetic
  in number fields}, CRYPTO '97: Proceedings of the 17th Annual International
  Cryptology Conference on Advances in Cryptology (London, UK),
  Springer-Verlag, 1997, pp.~385--394.

\bibitem{BWKeyEx}
J.~Buchmann and H.~C. Williams, \emph{A key-exchange system based on imaginary
  quadratic fields}, Journal of Cryptology \textbf{1} (1988), 107--118.

\bibitem{BWKeyExReal}
\bysame, \emph{A key-exchange system based on real quadratic fields}, CRYPTO
  '89, Lecture Notes in Computer Science, vol. 435, 1989, pp.~335--343.

\bibitem{Canfield}
E.R. Canfield, P.~Erd\"{o}s, and C.~Pomerance, \emph{On a problem of oppenheim
  concerning `factorisatio numerorum'}, J. Number Theory \textbf{17} (1983),
  1--28.

\bibitem{cassel}
J.~Cassels, \emph{An introduction to the geometry of numbers}, Classics in
  Mathematics, Springer-Verlag, Berlin, 1997, Corrected reprint of the 1971
  edition.

\bibitem{cohen}
H.~Cohen, \emph{A course in computational algebraic number theory}, Graduate
  Texts in Mathematics, vol. 138, Springer-Verlag, 1991.

\bibitem{Enge}
A.~Enge and P.~Gaudry, \emph{An {L} (1/3 + {$\epsilon$}) algorithm for the
  discrete logarithm problem for low degree curves}, EUROCRYPT '07: Proceedings
  of the 26th annual international conference on Advances in Cryptology
  (Berlin, Heidelberg), Lecture Notes in Computer Science, Springer-Verlag,
  2007, pp.~379--393.

\bibitem{Enge2}
{A}ndreas {E}nge, {P}ierrick {G}audry, and {E}mmanuel {T}hom{\'e}, \emph{{A}n
  ${L} (1/3)$ {D}iscrete {L}ogarithm {A}lgorithm for {L}ow {D}egree {C}urves}.

\bibitem{janusz}
G.~Janusz, \emph{Algebraic number fields}, Pure and applied mathematics,
  vol.~55, New York, Academic Press, 1973.

\bibitem{LLL}
A.K. Lenstra, H.W. Lenstra, and L.~Lov\'{a}sz, \emph{Factoring polynomials with
  rational coefficients}, Mathematische Annalen \textbf{261} (1982), 515--534.

\bibitem{JaSto}
Jr. M.~Giesbrecht, M.~Jacobson and A.~Storjohann, \emph{Algorithms for large
  integer matrix problems}, Proc. AAECC-14 (S.~Boztas and I.~E. Shparlinski,
  eds.), Lecture Note in Computer Science, vol. 2227, Springer-Verlag, 2001.

\bibitem{arb_dim_3}
A.~Meyer, S.~Neis, and T.~Pfahler, \emph{First implementation of cryptographic
  protocols based on algebraic number fields}, ACISP '01: Proceedings of the
  6th Australasian Conference on Information Security and Privacy (London, UK),
  Springer-Verlag, 2001, pp.~84--103.

\bibitem{mignotte}
M.~Mignotte, \emph{An inequality about factors of polynomials}, Mathematics of
  Computation \textbf{28} (1974), 1153--1157.

\bibitem{seysen}
M.~Seysen, \emph{A probabilistic factorization algorithm with quadratic forms
  of negative discriminant}, Mathematics of computation \textbf{48} (1987),
  757--780.

\bibitem{Vdl}
U.~Vollmer, \emph{Asymptotically fast discrete logarithms in quadratic number
  fields}, Algorithmic Number Theory --- ANTS-IV, Lecture Notes in Computer
  Science, vol. 1838, 2000, pp.~581--594.

\end{thebibliography}

\end{document}